\def\draft{n}
\documentclass{amsart}
\usepackage[headings]{fullpage}
\usepackage{amssymb,epic,eepic,epsfig}%,amsbsy,amsmath,bbold}
\usepackage{graphicx}
\usepackage{texdraw}
\usepackage{url}
\usepackage[bookmarks=true,%
    colorlinks=true,%
    linkcolor=blue,%
    citecolor=blue,%
    filecolor=blue,%
    menucolor=blue,%
    urlcolor=blue,%
    breaklinks=true]{hyperref}

\newtheorem{theorem}{Theorem}[section]
\newtheorem{proposition}{Proposition}[section]
\theoremstyle{definition}
\newtheorem{lemma}[proposition]{Lemma}
\newtheorem{definition}[proposition]{Definition}
\newtheorem{remark}[proposition]{Remark}
\newtheorem{corollary}[proposition]{Corollary}

\newtheorem{question}[proposition]{Question}

\newcommand{\qbinom}[2]{\genfrac{[}{]}{0pt}{}{#1}{#2}}

\def\printname#1{
        \if\draft y
                \smash{\makebox[0pt]{\hspace{-0.5in}
                        \raisebox{8pt}{\tt\tiny #1}}}
        \fi
}

\newcommand{\psdraw}[2]
         {\begin{array}{c} \hspace{-1.3mm}
        \raisebox{-4pt}{\epsfig{figure=draws/#1.eps,width=#2}}
        \hspace{-1.9mm}\end{array}}

\newlength{\standardunitlength}
\catcode`\@=11
\long\def\@makecaption#1#2{%
     \vskip 10pt

\setbox\@tempboxa\hbox{%\ifvoid\tinybox\else\box\tinybox\fi
       \small\sf{\bfcaptionfont #1. }\ignorespaces #2}%
     \ifdim \wd\@tempboxa >\captionwidth {%
         \rightskip=\@captionmargin\leftskip=\@captionmargin
         \unhbox\@tempboxa\par}%
       \else
         \hbox to\hsize{\hfil\box\@tempboxa\hfil}%
     \fi}
\font\bfcaptionfont=cmssbx10 scaled \magstephalf
\newdimen\@captionmargin\@captionmargin=2\parindent
\newdimen\captionwidth\captionwidth=\hsize
\catcode`\@=12

\def\lbl#1{\label{#1}\printname{#1}}

\def\BN{\mathbb N}
\def\BZ{\mathbb Z}

\def\BQ{\mathbb Q}

\def\BC{\mathbb C}

\def\a{\alpha}

\def\l{\lambda}
\def\Ga{\Gamma}

\def\s{\sigma}

\def\ga{\gamma}

\def\la{\langle}
\def\ra{\rangle}

\def\Ga{\Gamma}

\def\b{\beta}

\def\s{\sigma}

\def\longto{\longrightarrow}

\def\pt{\partial}

\def\z{\zeta}

\def\fsl{\mathfrak{sl}}
\def\rra{\ra\hspace{-2pt}\ra}
\def\lla{\la\hspace{-2pt}\la}

\begin{document}

%%%%%%%%%%%%%%%%%%%%%%{page1}

\title[Asymptotics of quantum spin networks at a fixed root of unity]{
Asymptotics of quantum spin networks at a fixed root of unity}
\author{Stavros Garoufalidis}
\author{Stavros Garoufalidis}
\address{School of Mathematics \\
         Georgia Institute of Technology \\
         Atlanta, GA 30332-0160, USA \newline 
         {\tt \url{http://www.math.gatech.edu/~stavros }}}
\email{stavros@math.gatech.edu}
\author{Roland van der Veen}
\address{Korteweg de Vries Institute for Mathematics \\
         University of Amsterdam \\
         P.O. Box 94248, 1090 GE Amsterdam, The Netherlands  \newline 
         {\tt \url{http://www.science.uva.nl/$\sim$riveen}}}
\email{r.i.vanderveen@uva.nl}

\thanks{S.G. was supported in part by NSF. \\
\newline
1991 {\em Mathematics Classification.} Primary 57N10. Secondary 57M25.
}

\date{November 3, 2010 }

%\dedicatory{\large{\bf Preliminary notes. Please do not 
%distribute under any circumstances!}}

\begin{abstract}
A classical spin network consists of a ribbon graph (i.e.,
an abstract graph with a cyclic ordering
of the vertices around each edge) and an admissible coloring of 
its edges by natural numbers. The standard evaluation of a spin network
is an integer number. In a previous paper, we proved an existence theorem
for the asymptotics of the standard evaluation of an arbitrary classical 
spin network when the coloring of its edges are scaled by a large natural 
number. In the present paper, we extend the results to the case of an
evaluation of quantum spin networks of arbitrary valency 
at a fixed root of unity. As in the classical case, our proofs use the
theory of $G$-functions of Andr\'e, together with some new results
concerning holonomic and $q$-holonomic sequences of Wilf-Zeilberger.   
\end{abstract}

\maketitle

\tableofcontents

\section{Introduction}
\lbl{sec.intro}

\subsection{History}
\lbl{sub.history}

Spin networks originally arose from calculations of angular momentum in 
quantum mechanics \cite{VMK}. 
They were formalized in the sixties by Penrose \cite{Pe1,Pe2} in attempt 
to quantize gravity combinatorially. 
Similar ideas were developed by Ponzano and Regge who related the 
semi-classical expansion of 3D gravity
to the Regge action, \cite{BL1,BL2,PR,Wi}. In recent years, 
spin networks have played an important role in the development of loop 
quantum gravity; see 
\cite{BC,EPR,Pz,Ro} and references therein. In the eighties, quantum spin 
networks were used by Reshetikhin, Turaev and Viro
as building blocks of combinatorially defined invariants
of knotted 3-dimensional objects; see \cite{Tu,TV}. Those topological 
invariants that are collectively called TQFT generalize the famous
Jones polynomial of a knot \cite{J}.

A key problem in classical and quantum spin networks is their asymptotic
behavior for large spins. In the case of the $6j$-symbols Ponzano-Regge
conjectured the leading term of an explicit asymptotic expansion and gave
ample numerical evidence; see \cite{PR}. The Ponzano-Regge conjecture 
was proven by Roberts \cite{Rb} using methods of geometric quantization.
The existence of a general asymptotic expansion for all classical spin
networks (to all orders in perturbation theory, in a constructive way)
was recently obtained by the authors in \cite{GV}. The method of \cite{GV}
was to convert the problem of asymptotic expansions to questions in
Algebraic Geometry and Number Theory, and use the highly developed theory of 
$G$-functions. An important part in this conversion was the use of 
holonomic functions, developed by Zeilberger.

\subsection{Classical spin networks}
\lbl{sub.cspin}

A {\em classical spin network} consists of a ribbon graph $\Ga$ (i.e.,
an abstract graph with a cyclic ordering of the vertices around 
each edge) and an admissible coloring $\ga$ of its edges by natural numbers. 
The standard evaluation $\la \Ga,\ga \ra$ of a spin network is an 
integer number. In a previous paper \cite{GV}, 
we proved an existence theorem for the asymptotics of the sequence
$\la \Ga, n\ga \ra$ of an arbitrary trivalent classical spin network when the 
coloring of its edges is scaled by a large natural 
number $n$. In addition, we presented several ways (algebro-geometric,
combinatorial, and numerical) for computing the asymptotics of 
$\la \Ga, n\ga \ra$ for large $n$. 

The goal of the present paper is to extend the results of \cite{GV}
to the case of quantum spin networks 
at a fixed root of unity. As in the case of classical spin networks, 
our proofs use the theory of $G$-functions of Andr\'e. There are two new
ingredients which allow us to use the theory of $G$-functions. They
 come from the theory of holonomic and $q$-holonomic
functions of Wilf-Zeilberger:
\begin{itemize}
\item[(a)]
Theorem \ref{thm.der} which states that the derivative of a $q$-holonomic 
sequence with respect to $q$ is $q$-honolomic.
\item[(b)]
Theorem \ref{thm.qholo} which states that the evaluation
of a $q$-holonomic sequence at a fixed root of unity is holonomic.  
\end{itemize}

\subsection{Quantum spin networks}
\lbl{sub.qspin}

Quantum spin networks differ from their classical versions in two ways:
\begin{itemize}
\item[(a)] the underlying graphs $\Ga$ are knotted 
(i.e., embedded in $S^3$) and not abstract, 
\item[(b)] their evaluations are polynomials of $q$, and not 
simply integer numbers.
\end{itemize}
Recall that a {\em knotted ribbon graph} is an embedded framed graph in 
$S^3$, of arbitrary valency,
together with a cyclic ordering of the edges around each vertex. We will 
restrict ourselves to integer framings
so that thickening the graph gives rise to an orientable surface. A {\em 
quantum spin network} consists of 
a knotted ribbon graph $\Ga$ 
together with a pair of functions $\ga=(\ga_E,\ga_V)$. Here 
$\ga_E: \mathrm{Edges}(\Ga)
\longto\BN$ is an admissible coloring $\ga$ of its edges, and 
$\ga_V: \mathrm{Vert}(\Ga)\longto\{\mathrm{projectors}\}$ 
a choice of a local projector at each vertex of $\Ga$. 
Local projectors are explained in detail in Section 
\ref{sec.eval}, and can be ignored in case $\Ga$ is a trivalent graph.

The {\em standard evaluation} $\la \Ga,\ga \ra(q)$ of a quantum spin network 
is a rational function of $q^{1/4}$; see \cite{Tu}. 
Recently, Costantino \cite{Co} obtained an integrality
result for the standard evaluation of a quantum spin network, namely:
\begin{equation}
\lbl{eq.costa}
\la \Ga,\ga \ra(q) \in q^{\frac{n}{4}}\BZ[q^{\pm \frac{1}{2}}]
\end{equation} 
For some $n$ depending on the network. Our first result deals with fixing 
the fractional power of $q$.

\begin{proposition}
\lbl{prop.1}
For every quantum spin network $(\Ga,\ga)$ there exist a $\BZ/4\BZ$-valued
quadratic form $Q(\ga)$ such that
\begin{equation}
\lbl{eq.0frame}
\la \Ga,\ga\ra(q)=q^{\frac{1}{4} Q(\ga)} \lla \Ga,\ga\rra(q), 
\qquad \lla \Ga,\ga\rra(q) \in \BZ[q^{\pm 1}].
\end{equation}
\end{proposition}
The above proposition allows us to define a modified evaluation  
$\lla \Ga,\ga \rra(\zeta)\in \BC$ that can be evaluated a 
{\em fixed root of unity} $\zeta$. 
To make $\lla \Ga,\ga \rra$ well defined we use the convention that 
$0\leq Q(\ga)\leq 3$.
When $\zeta=1$, it follows from the definitions that $\lla \Ga,\ga \rra(1)$ 
equals the classical
evaluation and is hence independent of the embedding of $\Ga$ in 3-space.
The goal of this paper is to extend the results of \cite{GV} to the 
case of the standard evaluation of a quantum spin network at a {\em fixed}
root of unity $\zeta$.

\subsection{Sequences of Nilsson type and $G$-functions}
\lbl{sub.prelim}

To state our results, we need to recall what is a $G$-{\em function} and what
is a sequence of {\em Nilsson type}. Sequences of Nilsson type 
are discussed in detail in \cite{Ga4}. We recall some definitions here
for the convenience of the reader.

\begin{definition}
\lbl{def.nilsson}
We say that a sequence $(a_n)$ is of 
{\em Nilsson type} if it has an {\em asymptotic expansion} of the form 
\begin{equation}
\lbl{eq.nilsson}
a_n \sim \sum_{\l,\a,\b} \l^{n} n^{\a} (\log n)^{\b} S_{\l,\a,\b}
h_{\l,\a,\b}\left(\frac{1}{n}\right)
\end{equation}
where 
\begin{itemize}
\item[(a)]
the summation is over a finite set of triples $(\l,\a,\b)$
\item[(b)]
the {\em growth rates} $\l$ are algebraic numbers of equal absolute value,
\item[(c)]
the {\em exponents} $\a$ are rational and the {\em nilpotency exponents} 
$\b$ are natural numbers,
\item[(d)]
the {\em Stokes constants} $S_{\l,\a,\b}$ are complex numbers,
\item[(e)]
the formal power series $h_{\l,\a,\b}(x) \in K[[x]]$ 
are Gevrey-1 (i.e., the coefficient of $x^n$ is bounded 
by $C^n n!$ for some $C>0$),
\item[(f)]
$K$ is a  {\em number field} generated by the coefficients of 
$h_{\l,\a,\b}(x)$ for all $\l,\a,\b$. 
\end{itemize}
\end{definition}

\begin{definition}
\lbl{def.Gfunction}
We say that series $G(z)=\sum_{n=0}^\infty a_n z^n$ is a {\em $G$-function}
if 
\begin{itemize}
\item[(a)]
the coefficients $a_n$ are algebraic numbers,
\item[(b)]
there exists a constant $C>0$ so that for every $n \in \BN$
the absolute value of every Galois conjugate of $a_n$ is less than or equal to 
$C^n$, 
\item[(c)]
the common denominator of $a_0,\dots, a_n$ is less than or equal 
to $C^n$, (where the common denominator $d$ of $a_0,\dots, a_n$ is 
the least natural number such that $d a_i$ is an algebraic integer for 
$i=1,\dots,n$),
\item[(d)]
$G(z)$ is holonomic, i.e., it satisfies a linear differential equation
with coefficients polynomials in $z$.
\end{itemize}
\end{definition}

The following connection between sequences of Nilsson type and $G$-functions
was observed in  \cite[Prop.2.5]{Ga3}.

\begin{theorem}
\lbl{thm.Gtaylor}\cite[Prop.2.5]{Ga3}
If $G(z)=\sum_{n=0}^\infty a_n z^n$ is a $G$-function, then $(a_n)$ is
a sequence of Nilsson type.
\end{theorem}
The above existence theorem has a valuable, effective corollary.

\begin{corollary}
\lbl{cor.Gtaylor}
If $G(z)=\sum_{n=0}^\infty a_n z^n$ is a $G$-function, and suppose that we
are given either (a) a linear differential equation for $G(z)$ with
coefficients in $L[x]$ for some number field $L$, or (b) a linear
recursion for $(a_n)$ with coefficients in $L[n]$, then one can effectively 
compute $\l,\a,\b$ and $h_{\l,\a,\b}(x) \in K[[x]]$ for a number field $K$
such that the asymptotic expansion \eqref{eq.nilsson} holds, for some
{\em unknown} Stokes constants $S_{\l,a\,\b}$.
\end{corollary}
In other words, a linear recursion for $(a_n)$ or a linear differential
equation for $G(z)$ allows us to compute the asymptotic expansion
\eqref{eq.nilsson} to all orders in $1/n$, up to a finite set of 
unknown Stokes constants. For explicit illustrations of the above corollary,
see \cite{FS,WZ2} and also \cite[Sec.10]{GV} where the authors discuss
in detail effective asymptotic expansions of evaluations of classical
spin networks.

\subsection{Statement of our results}
\lbl{sub.results}

Let $f^{(r)}(q)=d^r/dq^r f(q)$ denote the $r$-th derivative of a Laurent
polynomial $f(q) \in \BZ[q^{\pm 1}]$.

\begin{theorem}
\lbl{thm.Nilsson}
For every quantum spin network $(\Ga,\ga)$, every complex root of unity
$\zeta$ and every natural number $r \geq 0$,
the sequence $\lla \Ga,n \ga \rra^{(r)}(\zeta)$ is of Nilsson type.
\end{theorem}

Theorem \ref{thm.Nilsson} follows from Theorem \ref{thm.Gtaylor} and the
following theorem.

\begin{theorem}
\lbl{thm.Gfun}
For every quantum spin network $(\Ga,\ga)$, every complex root
of unity $\zeta$ and every natural number $r$, the generating function
\begin{equation}
\lbl{eq.FGa}
F_{\Ga,\ga,\zeta,r}(z)=\sum_{n=0}^\infty \lla \Ga, n\ga \rra^{(r)}(\zeta) 
z^n
\end{equation} 
is a $G$-function.
\end{theorem}

Theorem \ref{thm.Gfun} follows from Theorems \ref{thm.der},
\ref{thm.qholo} and \ref{thm.expbound} below, which involve structural 
properties of the
classes of holonomic and $q$-holonomic sequences and may be of independent 
interest. To state them, recall that a sequence $(f_n)$ of complex
numbers is {\em holonomic} if it satisfies a linear
recursion of the form
\begin{equation}
c_d(n) f_{n+d} + \dots + c_0(n) f_n=0
\end{equation}
for all $n$ where $c_j(n) \in K[n]$ are polynomials in $n$ 
with coefficients in a number field $K$ for $j=0,\dots,d$
with $c_d \neq 0$. Likewise, a sequence $(f_n(q))$ of rational functions of
$q$ is $q$-{\em holonomic} if it satisfies a linear
recursion of the form
\begin{equation}
c_d(q^n,q) f_{n+d}(q) + \dots + c_0(q^n,q) f_n(q)=0
\end{equation}
for all $n$ where $c_j(u,v) \in K[u,v]$ are polynomials in
two variables $u$ and $v$ for $j=0,\dots,d$ with $c_d \neq 0$. 
Holonomic and $q$-holonomic sequences were studied in detail by Zeilberger; 
see \cite{Z,WZ1}.

\begin{theorem}
\lbl{thm.der}
The derivative with respect to $q$ of a $q$-holonomic sequence is 
$q$-holonomic.
\end{theorem}

\begin{theorem}
\lbl{thm.qholo}
For every $q$-holonomic sequence $f_n(q) \in \BZ[q^{\pm 1}]$ and every
complex root of unity $\zeta$, 
the evaluation $f_n(\zeta)$ is a holonomic
sequence. % which is exponentially bounded.
\end{theorem}
Let us make some remarks.

\begin{remark}
\lbl{rem.qholo1}
Theorem \ref{thm.qholo} fails when $\zeta$ is
not a complex root of unity. For example, $f_n(q)=q^{n^2}$ is $q$-holonomic
and satisfies the recursion $f_{n+1}(q)-q^{2n+1} f_n(q)=0$. On the other hand,
$f_n(\omega)$ is holonomic only when $\omega$ is a complex root of unity. 
Theorem \ref{thm.qholo} is another manifestation
of the importance of roots of unity in Quantum Topology. See also Section 
\ref{sec.proofs}.
\end{remark}
 
\begin{remark}
\lbl{rem.qholo2}
Theorems \ref{thm.der} and \ref{thm.qholo} presumably hold 
for for multi-variable $q$-holonomic sequences 
$f_{n_1,\dots,n_r}(q) \in \BZ[q^{\pm 1}]$. Multi-variable holonomic and $q$-holonomic
sequences were introduced and studied in \cite{Z,WZ1}. 
In the present paper, we will not use them. However, the reader should keep in
mind that for every quantum spin network $(\Ga,\ga)$ the multi-variable
sequence $\ga\mapsto \la \Ga,\ga\ra(q)$ is $q$-holonomic. This follows
from Section \ref{sub.recoupling} where it is shown that 
$\la \Ga,\ga\ra(q)$ is a multi-sum of a $q$-proper hypergeometric term. 
By the fundamental theorem of WZ-theory (see \cite{WZ1}),
$\la \Ga,\ga\ra(q)$ is $q$-holonomic in all $\ga$-variables. 
\end{remark}

\begin{remark}
\lbl{rem.qholo3}
The $q$-holonomic sequence $f_n(\zeta)$ in Theorem \ref{thm.qholo}
need not be exponentially bounded. For example,
$$
f_n(q)=\prod_{k=1}^n\frac{1-q^k}{1-q} \in \BZ[q]
$$
satisfies the linear recursion
$$
(q-1) f_{n+1}(q)-(q^{n+1}-1)f_n(q)=0.
$$
Thus, $f_n(q)$ is $q$-holonomic. Its evaulation at $\zeta=1$ is given by
$f_n(1)=n!$ which is holonomic, since it satisfies the linear recursion
$$
f_{n+1}(1)-(n+1)f_n(1)=0
$$
However, $n!$ is not exponentially bounded. 
\end{remark}
Despite the above remark, the evaluations of quantum spin networks at
a fixed root of unity is exponentially bounded.

\begin{theorem}
\lbl{thm.expbound}
For every quantum spin network $(\Ga,\ga)$, every complex root
of unity $\zeta$ and every natural number $r$, there exists $C>0$ (which
depends on $\Ga,\ga$ and $r$) such that
\begin{equation}
|\lla \Ga, n\ga \rra^{(r)}(\zeta)| \leq C^n 
\end{equation} 
for all $n \in \BN$.
\end{theorem}

\subsection{Acknowledgment}
The paper was conceived during a conference on 
{\em Interactions Between Hyperbolic Geometry, Quantum Topology and 
Number Theory} held in Columbia University, New York, in June 2009.
The authors wish to thank the organizers, A. Champanerkar, O. Dasbach, 
E. Kalfagianni, I. Kofman,  W. Neumann and N. Stoltzfus 
for their hospitality. In addition, we wish to thank F. Costantino,
G. Kuperberg and D. Zagier for enlightening conversations.

\section{Evaluation of quantum spin networks}
\lbl{sec.eval}

\subsection{Local projectors}
\lbl{sub.projectors}

We start by defining classical spin networks of arbitrary valency using
the notion of a local projector. A {\em classical spin network} will be a 
pair $(\Ga,\ga)$ of an abstract ribbon graph $\Ga$ (of arbitrary valency), 
together with a pair of functions $\ga=(\ga_E,\ga_V)$ where 
$\ga_E: \mathrm{Edges}(\Ga)
\longto\BN$ is an admissible coloring $\ga$ of its edges, and 
$\ga_E: \mathrm{Vert}(\Ga)\longto\{\mathrm{projectors}\}$ 
a choice of a local projector at each vertex of $\Ga$. 
A local projector is defined as follows. Given an admissible coloring 
$(c_1,\dots,c_d)$ of the $d$ edges around a vertex $v$ of a ribbon graph, 
place a collection of $c_1+\dots + c_d$
points on a disk. A {\em local projector} is a planar way to
connect these points with {\em disjoint arcs} on the disk and 
{\em no $U$-turns}, as in the following example:

\begin{figure}[htpb]
$$
\psdraw{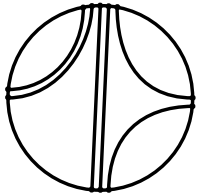}{1in}
$$
\caption{An example of a projector for a degree $4$ vertex with edge 
coloring $(2,5,2,3)$.}\lbl{fig.Projector}
\end{figure}
Note that if $p$ is a local projector and $n$ is a natural number, then
there is there is a canonical choice of a local projector $n p$ by cabling
each arc of $p$ into $n$ arcs.
In \cite[p.118]{Ku}, Kuperberg uses the term {\em clasped web space} 
$W(c_1,\dots,c_d)$ which has a basis the local projectors defined above.
Local projectors are a pictorial way to encode dual canonical bases for
$\fsl_2$, as was shown in \cite{FK}. However, clasped web spaces
and dual canonical bases do {\em not} coincide for $\fsl_3$; see \cite{KK}.
Finally note that 
there are alternative methods of defining multi-valent vertices,
for example \cite{Yo},\cite{Ba}. However our definition using 
projectors works without change for Lie algebras of arbitrary rank.

\subsection{The standard evaluation of a quantum spin network}
\lbl{sub.standard}

Recall from the introduction that a quantum spin network is a pair 
$(\Ga,\ga)$ where $\Ga$ is a framed ribbon graph embedded in $S^3$ and 
$\ga = (\ga_E,\ga_V)$ 
as in the classical case. We allow $\Ga$ to have multiple edges and 
loops and it may be disconnected. Knots and links are quantum spin 
networks with zero
vertices. However we restrict ourselves to integral framings of $\Ga$.
In this section we define the evaluation of quantum spin 
networks in terms of the Kauffman bracket. 
Recall that the {\em quantum integer} $[n]$ and the balanced 
{\em quantum factorial}
$[n]!$ is defined by
\begin{equation}
\lbl{eq.qinteger}
[n]=\frac{A^{2n}-A^{-2n}}{A^2-A^{-2}}, \qquad [n]!=\prod_{k=1}^n [k]
\end{equation}
where
\begin{equation}
\lbl{eq.Ap}
A^4 = q.
\end{equation}
The {\em Kauffman bracket} evaluation is determined by the following rules:
\begin{eqnarray*}
\psdraw{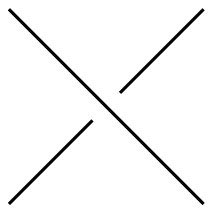}{0.3in} &=& A \psdraw{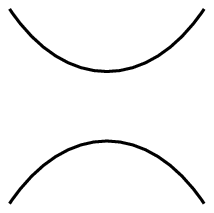}{0.3in} + A^{-1} 
\psdraw{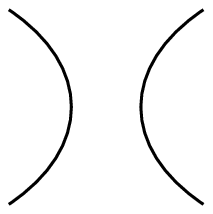}{0.3in} \\ \vspace{1cm}
\psdraw{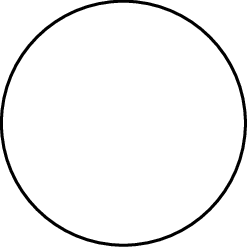}{0.2in} \cup \mathrm{D} &=& -[2]\cdot \mathrm{D}
\end{eqnarray*}
To emphasize the similarity to the evaluation of classical spin networks we 
use the explicit expression for the Jones-Wenzl idempotent
from \cite[Chpt.3]{KL}. This is defined as the following formal sum of braids:
$$
A^{b(b-1)}\sum_{\sigma \in S_b}A^{-3 \ell(\sigma)}\b_{\sigma}
$$
Here $b$ is the label of the edge, and for any permutation $\sigma$ we 
denote by $\b_{\sigma}$ the unique negative 
(with respect to orientation downwards) permutation braid corresponding 
to $\sigma$.
By $\ell(\sigma)$ we mean the minimal length of $\sigma$ written as a product 
of transpositions.
Note that we leave out the quantum factorial $[b]!$ that is used in \cite{KL}.

\begin{definition}
\lbl{def.evalP}
\rm{(a)} We say a quantum spin network is {\em admissible} if for every 
vertex $v$ the projector $\ga_V(v)$
matches the labels of the edges given by $\ga_E$.
\newline
\rm{(b)}
The evaluation $\la \Ga,\ga\ra^P$ of a quantum spin network $(\Ga,\ga)$
is defined to be zero if it is not admissible. An admissible quantum spin 
network is 
evaluated by the
following algorithm. 
\begin{itemize}
\item
Use the ribbon structure to thicken the vertices into disks and the 
edges into untwisted bands. 
\item
Replace each vertex $v$ by the pattern of the projector $\ga_V(v)$ and 
replace each edge by the linear combination of braids as shown in Figure 
\ref{fig.evaluation}. 
\item
Finally 
the resulting linear combination of links is evaluated calculating the 
Kauffman bracket.
\end{itemize}
\end{definition}   

\begin{figure}[htpb]
$$
\psdraw{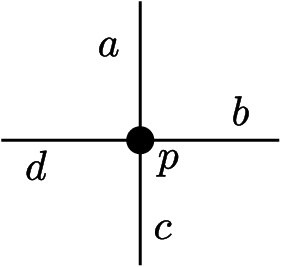}{1in}\quad \to\quad \psdraw{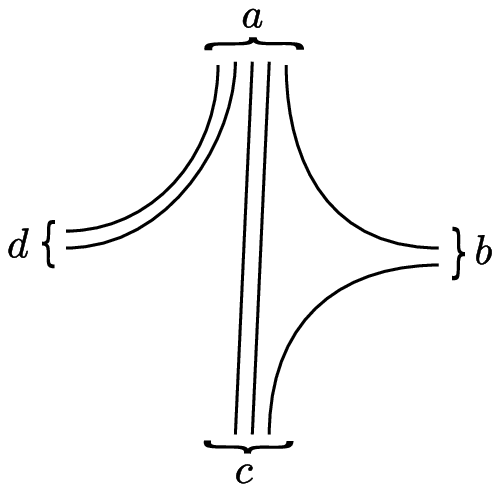}{1.6in} \hspace{0.5in} 
\psdraw{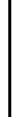}{0.25in}\quad \to\quad A^{b(b-1)}\sum_{\sigma \in S_b} A^{-3\ell(\sigma)}
\psdraw{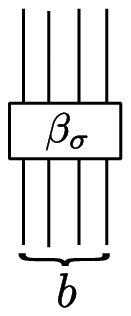}{0.4in}
$$
\caption{The rules for evaluating a quantum spin network. Replace 
vertices and edges 
to get a linear combination of links that is evaluated using the 
Kauffman bracket.}\lbl{fig.evaluation}
\end{figure}

In the case $A = -1$ this definition agrees with the Penrose evaluation 
defined in \cite{GV}.
For general $A$ the definition agrees with that given in \cite{MV,KL,CFS} 
except for the missing quantum factorial $[b]!$ in the denominator, one 
for every edge. The reason for leaving out this factor is that there is a 
better way to normalize (see Definition \ref{def.eval} below) which
results in Laurent polynomials in $q^{1/4}$ (as opposed to rational functions
in $q^{1/4}$) while the value of 
$\la \Ga ,n\ga \ra(q)$ at a fixed root of unity $q$ 
grows at most exponentially with respect to $n$, see Theorem \ref{thm.expbound}.

\begin{definition}
\lbl{def.eval}
The {\em standard evaluation} of a quantum spin network is defined by 
$$
 \la \Ga,\ga \ra = \frac{1}{[\mathcal{I}]!}\la \Ga,\ga \ra^P
$$
Here $[\mathcal{I}]! = \prod_v [\ga_V(v)]!$, where for a projector $[p]!$ 
means grouping 
all strands connecting the same two edges and forming the product of the 
quantum factorials of
these numbers. 
\end{definition}

Note that with this normalization subdividing an edge colored $a$ into two 
edges colored $a$ connected by a vertex whose third edge is colored $0$ 
does not change the value of the evaluation. This is the way in which the 
value of the unknot should be interpreted in order to obtain the correct 
value $(-1)^a[a+1]$ instead of its quantum factorial. 

We end this section by proving that any quantum spin network evaluation
actually reduces to a trivalent quantum spin network evaluation. Hence
all results previously known in the trivalent case extend to the general case.
 
\begin{lemma}
\lbl{lem.trivalent}
Let $(\Ga,\ga)$ be a quantum spin network. There exists a trivalent 
quantum spin network $(\Ga',\ga')$ such that for all $n\in \mathbb{N}$
$$
 \la \Ga,n\ga \ra =\la \Ga',n\ga'\ra 
$$
\end{lemma}

\begin{proof}
It is convenient to take a slightly different view of the evaluation 
algorithm
defined in \ref{def.evalP}. Instead of expanding out the linear combination 
at
every edge, we view the Jones-Wenzl idempotent as box with arcs coming out. 
After expanding the 
vertices
the whole network becomes a number of boxes connected by arcs. 
The idea is to add extraneous boxes and to reinterpret the result as the
evaluation of a trivalent quantum spin network.
The key property that makes this work is the fundamental fact that 
once one includes the factor $\frac{1}{[b]!}$ in a box with $b$ strands, 
it becomes
an idempotent \cite{MV}. Since we're not using this normalization we get that 
adding an extra box to $b$ parallel adjacent arcs coming out of a box
multiplies the evaluation by $[b]!$. This factor will cancel with the
the normalization factor due to the trivalent vertices that will be 
created in the process.

Let us first describe how to modify a single half-edge pointing into a 
multivalent vertex $v$ with projector $p$.
The arcs coming into $v$ from half edge $e$ are grouped into $a_j$ parallel 
arcs according to the projector $p$. To the box that already marks the 
beginning of $e$ we now 
add an extra stack of boxes as shown in the middle of Figure
\ref{fig.Trivalent}. This will multiply the Penrose evaluation of $\Ga$ by a 
factor 
$a_e = ([a_2]!...[a_n]!)([a_1+a_2...+a_{n-1}]!...[a_1+a_2]![a_1]!)$

\begin{figure}[htpb]
$$
\psdraw{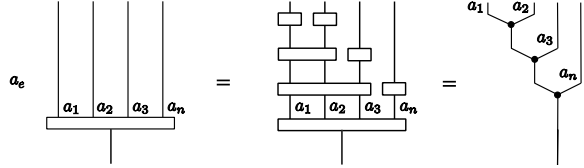}{3in}
$$
\caption{Turning an edge into a trivalent network.}\lbl{fig.Trivalent}
\end{figure}

Doing this for every half edge gives just enough boxes to split into a 
trivalent quantum spin network $\Ga'$ as
shown in the above figure. However the smallest boxes at the tip of the 
tree meet another equal sized box
on the other side which are not there in the evaluation of $\Ga'$. 
Therefore we have to remove those boxes again, which 
reduces the Penrose evaluation by exactly a factor 
$[\mathcal{I}]! = [\mathcal{I}(\Ga)]!$. Therefore,
$$
\la\Ga,\ga\ra^P\prod_e a_e =  [\mathcal{I}]!\la \Ga',\ga' \ra^P
$$

And hence passing to the standard normalization we find:

$$
\la\Ga,\ga\ra = \frac{\la\Ga,\ga\ra^P}{[\mathcal{I}]!} =  
\frac{\la \Ga',\ga' \ra^P}{\prod_e a_e} = \la \Ga',\ga' \ra
$$

concluding the proof.
\end{proof}

\subsection{Evaluation of spin networks using the shadow formula}
\lbl{sub.recoupling}

In this subsection we describe a way of evaluating spin networks in terms
of the shadow formula. We restrict ourselves to trivalent $\Ga$. Since 
$\Ga$ is supposed to be 
orientable, we can choose a blackboard framed diagram $D$ of $\Ga$.

The shadow formula expresses the evaluation in terms of 
a multi-dimensional sum of $1j$, $3j$ and $6j$-symbols. The latter are
the evaluations of three basic spin networks, shown in Figure \ref{fig.3j6j}.

\begin{figure}[htpb]
$$
\psdraw{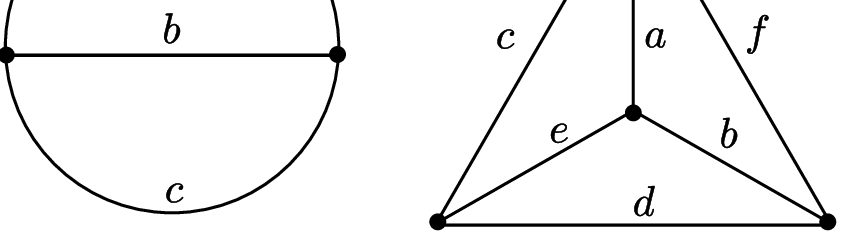}{3in}
$$
\caption{Three basic spin networks.}\lbl{fig.3j6j}
\end{figure}
\noindent
Let 
\begin{equation}
\lbl{eq.qmulti}
\qbinom{a}{a_1, a_2, \dots, a_r}=\frac{[a]!}{[a_1]! \dots [a_r]!}
\end{equation}
denote the quantum multinomial coefficient when $a_1+ \dots a_r=a$. 
The value of the $1j$, $3j$ and $6j$-symbols is given by the
following lemma of \cite{KL} (see also, \cite{MV,We}), using our 
normalization.

\begin{lemma}
\lbl{lem.136j}
\rm{(a)} We have
$$ 
\la \psdraw{1j}{0.2in}\, ,\,a \ra=(-1)^a[a+1]
$$
\rm{(b)} Let $(\Theta,\ga)$ denote the $\Theta$ spin network admissibly 
colored by $\ga=(a,b,c)$ as in Figure \ref{fig.3j6j}.
Then we have
\begin{equation}
\lbl{eq.3jsum}
\la \Theta, \ga \ra 
= (-1)^{\frac{a+b+c}{2}}[\frac{a+b+c}{2}+1]
%\qbinom{-2}{\frac{a+b+c}{2}} 
\qbinom{\frac{a+b+c}{2}}{\frac{-a+b+c}{2}, \frac{a-b+c}{2}, \frac{a+b-c}{2}}
\end{equation}
\rm{(c)}
Let $(\psdraw{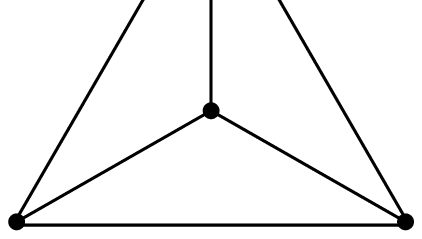}{0.18in},\ga)$ denote a tetrahedron labeled and 
oriented as in Figure \ref{fig.3j6j} and admissibly colored by
$\ga=(a,b,c,d,e,f)$. Then we have
\begin{equation}
\lbl{eq.6jsum}
\la \psdraw{tetra}{0.18in}, \ga \ra =
\sum_{k = \max T_i}^{\min S_j} (-1)^k [k+1]
%\qbinom{-2}{k}
\qbinom{k}{
S_1-k , S_2-k , S_3-k , k- T_1 , k- T_2 , k- T_3 , k- T_4}
\end{equation} 
where
$S_i$ are the half sums of the labels in the three quadrangular curves in 
the tetrahedron 
and $T_j$ are the half sums of the thee edges emanating from a given vertex. 
In other words, the $S_i$ and $T_j$ are given by
\begin{equation}
\lbl{eq.Si}
S_1 = \frac{1}{2}(a+d+b+c)\qquad S_2 = \frac{1}{2}(a+d+e+f) 
\qquad S_3 = \frac{1}{2}(b+c+e+f)
\end{equation}
\begin{equation}
\lbl{eq.Tj}
T_1 = \frac{1}{2}(a+b+e) \qquad T_2 = \frac{1}{2}(a+c+f)
\qquad T_3 = \frac{1}{2}(c+d+e) \qquad T_4 = \frac{1}{2}(b+d+f).
\end{equation}
\end{lemma}

In addition to the three basic spin networks above, we choose to consider a 
variant of the
tetrahedron that represents a crossing. The crossing can be reduced to a 
$6j$-symbol using
the half twist formula, \cite{MV}.

$$
\psdraw{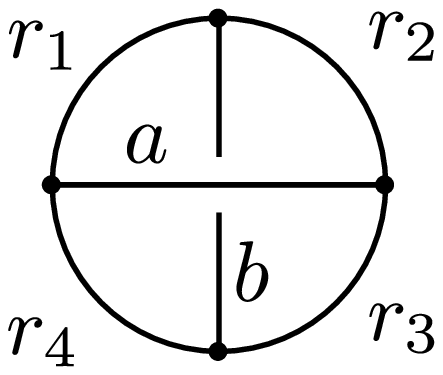}{0.4in} = 
(-1)^{\frac{-r_1-r_3+r_2+r_4}{2}}q^{\frac{-r_1(r_1+2)-r_3(r_3+2)+r_2(r_2+2)+r_4(r_4+2)}{8}}
\psdraw{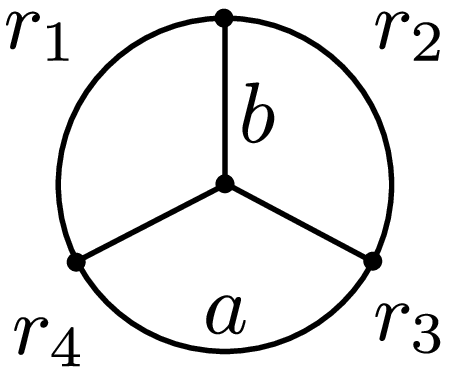}{0.4in}
$$

We are now ready to explain how to evaluate a spin network using the shadow 
formula. 
A blackboard framed diagram of $\Ga$ gives rise to a planar graph $D$, whose
edges are colored by $\ga$ and whose vertices are either vertices of $\Ga$ 
or crossings.
Let $V,E,F,C$ denote the sets of trivalent vertices, edges, faces and 
crossings ($4$-valent vertices) of $D$. 
We will express the evaluation of the quantum spin network as a sum over 
all admissible colorings of $F$ by natural numbers.
Here admissible means that for any two faces colored $r_1,e,r_2$ and 
separated by edge $e$, the
numbers $(r_1,r_2,\ga(e))$ satisfy the triangle inequalities. Moreover 
the outer face should get color $0$.
\begin{equation}
\lbl{eq.shadow}
\la \Ga,\ga\ra = \sum_{r} 
\prod_{f\in F}\psdraw{1j}{0.15in}\prod_{e\in E}\Theta^{-1}
\prod_{v \in V}\psdraw{tetra}{0.18in}\prod_{c \in C}\psdraw{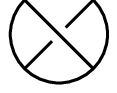}{0.18in}
\end{equation}
The sum is over admissible colorings $r$ of the faces $F$ and the labels 
of the
symbols in the formula are found from the coloring and $\ga$ as follows. 
The 
$1j$ symbol gets the color of the face it corresponds to. The theta gets 
the color of 
the edge it corresponds to plus the two colors of the faces it bounds. For 
every vertex the corresponding $6j$ symbol is colored by the three adjacent 
faces and the three adjacent edge labels. Finally every crossing is 
counted by a skew 
$6j$-symbol obtained by encircling the crossing and coloring according 
to the faces one crosses.
For a more elaborate discussion see \cite{Co}.

Consider the following simple example of a spin network evaluation using 
the shadow formula. 

\begin{figure}
$$
\psdraw{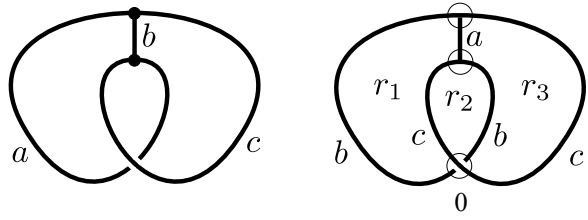}{3in}
$$
\caption{An example spin network $\Ga,\ga$ (left) and its diagram with 
encircled crossings and vertices and a coloring of the regions.}
\end{figure}

Note that one can find the contributing $6j$-symbols from the diagram by 
encircling the vertices and crossings. This is useful in
keeping track of the labels. Similarily one could also encircle the edges 
and faces to obtain the $3j$ and $1j$-symbols.
$$
\la\Ga,a,b,c\ra = 
\sum_{r_1,r_2,r_3} \frac{\psdraw{1j}{0.2in}^{r_1}\psdraw{1j}{0.2in}^{r_2}
\psdraw{1j}{0.2in}^{r_3}
\psdraw{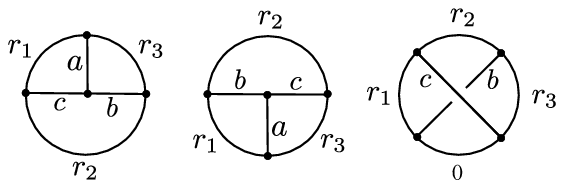}{2in}}{\Theta(0,b,r_1)\Theta(0,c,r_3)\Theta(r_1,c,r_2)\Theta(r_2,b,r_3)\Theta(r_1,a,r_3)}
$$

Since the sum is over admissible colors only, we see that in this special 
case there is quite some simplification. 
The fact that the outside color is $0$ forces $r_1 = b$, $r_2 = c$. Actually 
one can check that after substituting the 
formulas for the $6j$, $3j$ and $1j$-symbols one gets that the answer is 
$0$ unless $a = 0$ and $b=c$ in which case we get a twisted unknot.

\subsection{Integrality}
\lbl{sub.integrality}

In \cite{Co} it was shown that $\la \Ga, \ga\ra \in A^f\mathbb{Z}[A^{\pm 2}]$ 
for some $f$ depending on $(\Ga,\ga)$. In this subsection we use the shadow 
formula to extend this result to prove that 
actually $\la \Ga, \ga\ra \in A^{Q(\ga)}\mathbb{Z}[A^{\pm 4}]$.
Here $Q(\ga)$ is a quadratic form depending on $\Ga$ in the labels $\ga$, 
that takes values in $\BZ/4\BZ$. The proof below will indicate how to obtain 
an expression 
for $Q(\ga)$ from an admissible coloring of the faces of a diagram of $\Ga$. 
The form itself
is independent of the choice of diagram, and coloring. By identifying 
$\BZ/4\BZ$ with $\{0,1,2,3\}$
we have defined $\lla \Ga,\ga\rra= A^{-Q(\ga)}\la \Ga, \ga\ra$ precisely.

\begin{proof}(of Proposition \ref{prop.1}).
Note that by Lemma \ref{lem.trivalent} we can restrict ourselves to trivalent 
$\Ga$. Pick any blackboard framed diagram $D$ of $(\Ga,\ga)$ and consider 
applying the shadow formula \eqref{eq.shadow}. Since by Costantino's result 
\cite{Co}
we already know the evaluation is a polynomial in $A$ we can look at the 
degree of
the terms in the sum.

Now define $\phi:\mathbb{Z}[A^{\pm 1}]\to \mathbb{Z}/4\mathbb{Z}$ to be a 
function such that $P(a)\in A^{\phi(P)}\mathbb{Z}[A^{\pm 4}]$. As a first step
we calculate $\phi$ for the building blocks of the shadow formula.
By dividing the leading powers of the denominator in the formulas for the 
$6j$-symbol we see that:
\begin{eqnarray*}
\phi([k]) &=& 2k-2 \\
\phi(\Theta,a,b,c) &=& a^2+b^2+c^2 \\
\phi(\psdraw{tetra}{0.18in},\ga) &=& \sum_{i}\ga_{i}^2 + \sum_{i<j}\ga_i \ga_j 
\\
\phi(\psdraw{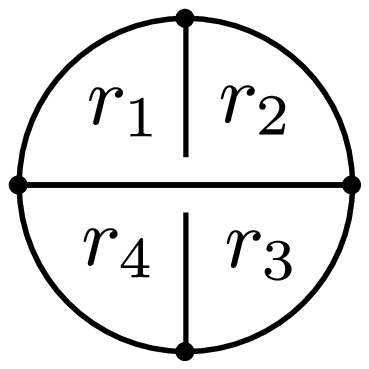}{0.5in}) &=& 
\frac{-r_1^2-r_3^2+r_2^2+r_4^2}{2}-r_1-r_3+r_2+r_4+\phi(\psdraw{tetra}{0.18in})
\end{eqnarray*}
It suffices to show that every term in the shadow 
formula has the same value of $\phi$ (modulo $4$).
In order to check this we consider the effect on $\phi$ of a term when we 
increase one of the variables $r$ by two. This is sufficient since a simple 
argument
shows that any state can be reached from any other state by repeatedly 
increasing
or decreasing one of the variables by two.

So let $\phi(r)$ be the value of $\phi$ on a particular summand in which we 
ignore the terms not containing $r$.
With respect to the region labeled $r$ we make a distinction between the 
edges and regions directly adjacent to it (notation:
$e | r$ or $r_i | r$) and those edges and regions transverse to the region 
$r$ (notation: $e \perp r$). See also Figure \ref{fig.Regions}.

\begin{figure}[htpb]
$$
\psdraw{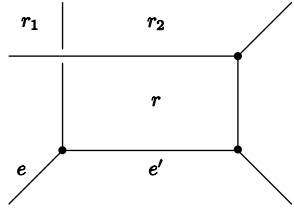}{1.5in}
$$
\caption{The region labeled $r$. The region $r_1$ the edge $e$ are transverse 
to $r$. The region $r_2$ and the edge $e'$ are adjacent to 
$r$.}\lbl{fig.Regions}
\end{figure}

Looking at the above formulae for $\phi$ and counting the crossings 
crossings next to $r$ 
with sign $\s$, 
we get the following expression for $\phi(r)$. Here we abused the notation to
make $r_i$ denote both the face and its value and $\s (r_i)$ is the sign of 
the corresponding
crossing between $r$ and $r_i$. 
$$
\phi(r) = r^2(\text{\#adjacent edges} +\text{\#adjacent vertices} 
+ \frac{1}{2}\text{\#crossings}+\text{\#unsigned crossings})+
$$
$$ 
r(2+2\sum_{e | r}\ga(e) +2\sum_{r_i | r}r_i + \sum_{e \perp r}\ga(e) 
+\sum_{r_i \perp r} r_i+\s(r_i))
$$

We will show that $\Delta = \phi(r+2)-\phi(r)$ is divisible by $4$. Since 
modulo $4$ we have 
$$
\Delta \equiv 2((r+1)\text{\#crossings}+\sum_{e \perp r}\ga(e) 
+\sum_{r_i \perp r} r_i+\s(r_i))\ \mod 4
$$
we need to check that the term in brackets is even. To do so we use 
admissibility and write congruences modulo $2$. By admissibility of the 
labels at the trivalent vertices and admissibility of the thetas 
coming from the edges in the shadow formula respectively we have:
$$
\ga(e) \equiv \ga(e') + \ga(e'')  \quad \quad r_i \equiv \ga(e_1) 
+ \ga(e_2) + r \quad \mod 2
$$
Here $e,e',e''$ are the edges at a vertex and $e_1,e_2,r,r_i$ are the edges 
and two opposite 
regions at a crossing. Summing all these identities we see that every edge 
adjacent to $r$ appears 
exactly twice, which shows that 
$$
\sum_{e \perp r}\ga(e) +\sum_{r_i \perp r} r_i \equiv r\text{\#crossings} \mod 2
$$
Finally $\sum_{r_i \perp r} \s(r_i) \equiv \text{\#crossings}$, 
therefore $\Delta$ is divisible by $4$ and the proof is complete.
\end{proof}

\section{Proof of Theorems \ref{thm.Gfun}, \ref{thm.der}, 
\ref{thm.qholo} and \ref{thm.expbound}}
\lbl{sec.proofs}

\subsection{Behavior of $q$-holonomic sequences under differentiation}
\lbl{sub.der}

In this section we prove Theorem \ref{thm.der}.

\begin{proof}
Consider a $q$-holonomic sequence $f_n(q) \in \BQ(q)$ that satisfies a linear
recursion relation
$$
\sum_{j=0}^d a_j(q^n,q) f_{n+j}(q)=0
$$
where $a_j(u,v) \in \BQ[u,v]$ for $j=0,\dots,d$. In the proof, the choice of 
coefficients $\BQ[u,v]$ versus $K[u,v]$ for a number field $K$ does not play an
important role, and therefore we assume $K=\BQ$. 
Differentiate with respect
to $q$. We obtain that
\begin{equation}
\lbl{eq.df1}
\sum_{j=0}^d a_j(q^n,q) f'_{n+j}(q)
+ n q^{n-1} \sum_{j=0}^d a_{j,u}(q^n,q) f_{n+j}(q) 
+ \sum_{j=0}^d a_{j,v}(q^n,q) f_{n+j}(q) 
=0
\end{equation}
where $a_{j,u}=\pt a_j/\pt u$, $a_{j,v}=\pt a_j/\pt v$ and
$f'(q)=d f(q)/dq$.
Recall now that the product and the $\mathbb{Q}[q]$-linear combination of 
two $q$-holonomic sequences is
$q$-holonomic; see \cite{Z,WZ1}. Lemma \ref{lem.nq} implies that for every 
$j$, the sequences $(n q^{n-1} a_{j,u}(q^n,q) f_{n+j}(q))$ and 
$(a_{j,v}(q^n,q) f_{n+j}(q))$ are $q$-holonomic. It follows that
the second and third sum in Equation \eqref{eq.df1} is $q$-holonomic, i.e.,
it satisfies a linear recursion relation with coefficients in $\BQ[q^n,q]$.
Substituting in this linear recursion the first sum of Equation 
\eqref{eq.df1}, it follows that the sequence $f'_n(q)$ is $q$-holonomic.
\end{proof}

\begin{lemma}
\lbl{lem.nq}
For every fixed integer $c$, the sequence $(n q^{c n})$
is $q$-holonomic.
\end{lemma}

\begin{proof}
It is easy to verify that the sequence $f_n(q)=n q^{c n}$ satisfies 
linear recursion relation
$$
f_{n+2}(q)-2q^c f_{n+1}(q)+q^{2c} f_n(q)=0
$$ 
%%% in fact, the coefficients of the above recursion are constant, and
%%% the eigenvalues are q^c, q^c (1 eigenvalue with multiplicity 2)
%the Equation
%$f_{n}(q) -n q^{c n}=0$. Replace $n$ by $n+1$ and subtract. It follows
%that
%$$
%f_{n+1}(q)-q^c f_n(q)=q^{c+c n}
%$$
%It follows that
\end{proof}

\subsection{Evaluations of $q$-holonomic sequences at roots of unity
are holonomic}
\lbl{sub.qholo}

Consider a $q$-holonomic sequence $f_n(q) \in \BZ[q^{\pm 1}]$ that satisfies
a linear recursion relation of the form
\begin{equation}
\lbl{eq.cj}
c_d(q^n,q) f_{n+d}(q) + \dots + c_0(q^n,q) f_n(q)=0
\end{equation}
for all $n$ where $c_j(u,v) \in \BQ[u,v]$ are polynomials in two variables
for $j=0,\dots,d$. The field $\BQ$ of coefficients does not matter here
and can be replaced by $\BC$ without any change. 
Fix a complex root of unity $\zeta$.
The idea of the proof consists of a lucky case, and a general case which
reduces to a lucky case after sufficient differentiation. The 
differentiation trick appears in an efficient algorithm to compute the Kashaev
invariant of some knots in \cite{GL1}, and was also inspired by conversations
of the first author with D. Zagier.
Let us first discuss the proof of Theorem \ref{thm.qholo} when 
$\zeta=1$. In the lucky case, the set
$$
S=\{j \, | 0 \leq j \leq d, \,\, c_j(1,1) \neq 0\}
$$
is non-empty. Expand \eqref{eq.cj} as a power series in $q-1$.
The vanishing of the constant term implies that $f_n(1)$ satisfies a 
non-trivial linear recursion with constant coefficients
$$
\sum_{j \in S} c_j(1,1) f_{n+j}(1)=0
$$
Thus, the sequence $(f_n(1))$ is holonomic. 
In the general case, there exists a unique natural number 
$s \geq 0$ such that $c_j(q^n,q)=\ga_j(q-1)^s + O((q-1)^{s+1})$ and some 
$\ga_j$ is nonzero. In other words, 
$c_j(q^n,q)$ vanish to order $s-1$ at $q=1$ for all $i$, and some $c_j(q^n,q)$ 
does not vanish to order $s$ at $q=1$.
In that case, observe that
$$
\sum_{j \in S'} c^{(s)}_j(1,1) f_{n+j}(1)=0
$$
where 
$$
S'=\{j \, | 0 \leq j \leq d, \,\, \ga_j \neq 0\} \neq \emptyset
$$
and $ c^{(s)}_j(1,1) \in \BQ[n]$. 
This concludes the proof of Theorem \ref{thm.qholo} when $\zeta=1$.

Suppose now that $\zeta$ is a complex root of unity of order $N$.
The problem is that $c^{(s)}_j(\z^{n},\z)$ is no longer a polynomial of $n$
even when $s=0$. To overcome this obstacle, 
we consider $n$ to be in a fixed arithmetic progression modulo $N$. 
In other words, 
fix $i$ with $0 \leq i \leq N-1$ and replace $n$ by $Nn+i$ in \eqref{eq.cj}.
Then, $c_j(q^{Nn+i},q)$ can be expanded in powers of $q-\z$ with coefficients
in $\BQ(\zeta)[n]$ for all $i$ and $j$. In other words, for all $i,j$ we
have $c_{ij}(q):=c_j(q^{Nn+i},q) \in \BQ(\zeta)[n][[q-\zeta]]$. Let us consider
the column vector $x_n=(f_{Nn}(\zeta),\dots,f_{Nn+N-1}(\zeta)) \in \BC^N$.
If $c_{i,j}(\zeta) \neq 0$, then $f_{Nn+i+j}(\zeta)$ is a $\BQ(\z)[n]$-linear 
combination of $f_{Nn+i+j'}(\zeta)$ for $j' \neq j$. Of course, $f_{Nn+i+j}(\zeta)$
is the $k$-th coordinate of $x_{n}$ where $k \equiv i+j \bmod N$.

Now we consider two cases. In the lucky Case 1, the following set
$$
S=\{(i,j) \in [0,N-1] \times [0,d] \, | \,\, c_{i,j}(\zeta) \neq 0\}
$$
is non-empty. $S$ defines a graph $G(S)$ defined as follows. It has vertex 
set $\{0,1,\dots,N-1\}$,
and an edge $e$ between vertices $k$ and $k'$ iff there exist $(i,j),(i,j') 
\in S$ such that $k \equiv i+j \bmod N$ and $k' \equiv i+j' \bmod N$.
We consider two subcases. In the very lucky Case 1.1, $G(S)$ is connected. The
above discussion implies that the vector $x_n$ and finitely many of its
translates, put together in a column vector $y_n$ satisfy a first order
linear recursion of the form
$$
y_{n+1}=A(n)y_n
$$ 
for a square matrix $A(n)$ with coefficients in the field $\BQ(\zeta)(n)$.
Lemma \ref{lem.cyclic} below implies that every 
coordinate of $y_n$ is holonomic. It follows that $(f_{Nn+i}(\zeta))$
is holonomic (with respect to $n$) for every fixed $i \in [0,N-1]$.
Since $(a_n)$ is holonomic if and only if $(a_{Nn+i})$ is holonomic for
all $i \in [0,N-1]$ (see \cite{PWZ}), it follows that $f_n(\zeta)$ is
holonomic in the very lucky Case 1.1. 

In the not-so-lucky Case 1.2, the set $S$ is non-empty but the graph $G(S)$
is disconnected. In that case, differentiate \eqref{eq.cj} once and consider
the pair of column vectors 
$x^{r}_n=(f_{Nn}^{(r)}(\zeta),\dots,f^{(r)}_{Nn+N-1}(\zeta)) \in \BC^N$ for $r=0,1$.
Consider also the set
$$
S^r=\{(i,j) \in [0,N-1] \times [0,d] \, | \,\, c_{i,j}^{(r)}(\zeta) \neq 0\}
$$
for $r=0,1$. $S^0 \cup S^1$ gives rise
to a graph $G(S^0 \cup S^1)$ with vertices pairs $(r,k)$ for $r=0,1$ and
$k=0,\dots,N-1$. There is an edge between $(r,k)$ and $(r',k')$ if there
exist $(i,j) \in S^{r}$, $(i,j') \in S^{r'}$ such that $k \equiv i+j \bmod N$ 
and 
$k' \equiv i+j' \bmod N$. 
%In addition, there is an edge between 
%$(0,k)$ and $(1,k')$ if there
%exist $(i,j) \in S^0$, $(i,j') \in S^1$ such that $k \equiv i+j \bmod N$ and 
%$k' \equiv i+j' \bmod N$. 
Note that by assumption, $S^0=S$ is non-empty, and there is a natural 
inclusion of $G(S)$ into $G(S^0 \cup S^1)$.
In the not-so-lucky Case 1.2, the graph $G(S^0 \cup S^1)$ is connected.
The argument of Case 1.1 shows that $f_n(\zeta)$ is holonomic.
If $G(S^0 \cup S^1)$ is disconnected, differentiate \eqref{eq.cj} once more.
It is easy to see that after a finite number $t$ of differentiations, the graph
$G(S^0 \cup \dots S^t)$ will be connected and consequently $f_n(\zeta)$
is holonomic. Differentiating \eqref{eq.cj} $t$ times and using induction
implies that $f_n(\zeta)$ is holonomic.

This concludes the proof of Theorem \ref{thm.qholo} when $S$ is non-empty.
In the unlucky Case 2 where all functions $c_{i,j}(q)$ vanish to first
order at $q=\zeta$, there is a natural number $s$ such that 
$c_{i,j}^{(s')}(\zeta)=0$ for all $s' < s$ and all $i,j$ and in addition
$c_{i,j}^{(s)}(\zeta) \neq 0$ for some $i,j$. In that case, replace $S$ by 
the set 
$$
S'=\{(i,j) \, | 0 \leq j \leq d, \,\, 0 \leq i < N-1, \,\,
c_{i,j}^{(s)}(\zeta) \neq 0\}
$$
and repeat the above proof.

\begin{lemma}
\lbl{lem.cyclic}
Consider a sequence of vectors $x_n \in \BC^l$ that satisfy a first order 
linear recursion of the form
$$
x_{n+1}=A(n) x_n
$$
where $A(n)$ is an $l\times l$ matrix with coefficients in $\BC[n]$. Then
every coordinate of $x_n$ is a holonomic sequence.
\end{lemma}

\begin{proof}
Consider the operators $\nu$ and $N$ which act on a sequence $(a_n)$ by
$$
(\nu a)_n=n a_n, \qquad (N a)_n=a_{n+1}.
$$
These operators satisfy the commutation relation $N\nu=\nu N+N$ and generate
an Ore ring $R=\BC[\nu]\la N \ra$. The matrix $A(n)$ defines an $R$-module $M$
as in \cite{VPS}. The lemma follows from the {\em cyclic vector theorem} 
applied to the module $M$; see \cite[Prop.2.9]{VPS} and also \cite{Kz}.
The proof is constructive and can be implemented for example in \cite{AB}.
\end{proof}

This finishes the proof of Theorem \ref{thm.qholo}.
\qed

Note that Theorem \ref{thm.qholo} fails when $q$ is not a root of unity. 
For example, consider the $q$-holonomic sequence $a_n(q) = q^{n^2}$ and 
fix $q = \omega$, a complex number which is not a root of unity.
Suppose that the sequence $b_n = \omega^{n^2}$ is holonomic. So for
all natural numbers $n$ we have
$$
\sum_{k=0}^d c_k(n) b_{n+k} =0
$$  
where $c_j(n)\in\mathbb{Q}[n]$.
Dividing by $b_n$ we get
$$
\sum_{k=0}^dc_k(n)\omega^{2nk+k^2} =0
$$  
Collecting the coefficients of a fixed power of $n$, we find that there are 
$C_1\hdots C_D\in \mathbb{C}$ such that
$$
\sum_{k=0}^D C_k \omega^{nk} = 0
$$
In other words, for all $n$, $\omega^n$ is a root of 
the polynomial $\sum_{k=0}^D C_k x^k$. It follows that $\omega^m =1$ for some 
$m$.

\subsection{Proof of Theorem \ref{thm.expbound}}
\lbl{sub.thm.expbound}

In this section we prove Theorem \ref{thm.expbound}. Recall from
Remark \ref{rem.qholo3} that this theorem is special to the quantum spin network 
evaluations and not valid for general $q$-holonomic sequences.

The  main ideas that go into the proof of Theorem 
\ref{thm.expbound} are the following: 
\begin{itemize}
\item[(a)]
its validity for the local building block of 
quantum spin networks (namely the quantum $6j$-symbol), 
\item[(b)]
The shadow formula \eqref{eq.shadow} for the evaluation of an arbitrary quantum spin network
where the summand is roughly a product of quantum $6j$-symbols divided by
quantum $3j$-symbols,
\item[(c)]
a lemma for the growth rate of the coefficients of inverse quantum factorials,
in the spirit of \cite[Thm.15]{GL1}.
\end{itemize}
Let us begin with the first ingredient.

\begin{lemma}
\lbl{lem.expbound}
Theorem \ref{thm.expbound} holds for the quantum $6j$-symbols.
\end{lemma}

\begin{proof}
Equation \eqref{eq.6jsum} shows that 
$f_n(q)=\la \psdraw{tetra}{0.18in}, n\ga \ra$ is a Laurent polynomial
in $q^{1/2}$ which is written as a 1-dimensional sum.
 Let $||h(q)||_1$ denote the sum of the absolute values of 
a Laurent polynomial in $q^{1/4}$. It is well known that the quantum
binomial coefficient \eqref{eq.qmulti} is a Laurent polynomial in $q^{1/2}$
with nonnegative integer coefficients. Moreover, its evaluation at $q=1$
is the usual multinomial coefficient. The sum of the latter is at most
$r^a$. It follows that
$$
||\qbinom{a}{a_1, a_2, \dots, a_r}||=\frac{a!}{a_1! \dots a_r!} \leq r^a.
$$
Equation \eqref{eq.6jsum} and the above inequality conclude the proof
of the lemma.
\end{proof}
Now consider an arbitrary quantum spin network $(\Ga,\ga)$
and its evaluation $f_N(q)=\la \Ga,N \ga\ra(q)$.
Viewing $1j$-symbols as special cases of $6j$-symbols and crossings
as a power of $q$ times a $6j$-symbol, the shadow state-sum formula \eqref{eq.shadow}
expresses $f_N(q)$ by 
\begin{equation}
\lbl{eq.fnk}
f_N(q)=\sum_{k \in N P} \frac{M(k)}{D(k)}
\end{equation}
where the summation is over the set of lattice points of $N P$ (for some
rational convex polytope $P$) and
$M(k)$ is a product a power of $q$ and of quantum $6j$-symbols at linear forms of $k$, and
$D(k)$ is a product of quantum factorials at linear forms of $k$.
The fact that $D$ is a product of quantum factorials only follows from the formula for the $3j$-symbol,
see Equation \eqref{eq.3jsum}.
We know that $f_N(q) \in \BZ[q^{\pm 1/4}]$ by Proposition \ref{prop.1}.
Theorem \ref{thm.qholo} implies that the exponents of $q$ in 
$\la \Ga,N \ga\ra(q)$ are bounded above and below by quadratic functions
of $N$. 

Let $(q)_n=\prod_{j=1}^n (1-q^j)$ for $n \in \BN \cup \{\infty\}$. It is
well-known that
$$
\frac{1}{(q)_\infty}=\sum_{k=0}^\infty p_k q^k
$$
where $p_k \in \BN$ is the number of partitions of $k$; see \cite{Aw}. 
Moreover, for every $k$ and $n$, the coefficient of $q^k$ in $1/(q)_n$
is a natural number which is at most $p_k$. In addition, we have
$$
p_k \leq e^{C k^{1/2}+o(1)}
$$
where $C=\pi \sqrt{2/3}$; see \cite{Aw}. It follows that the coefficient
of $q^n$ in $1/D(k)$ is exponentially bounded by $n^2$, and since
$n=O(N^2)$, the coefficient is exponentially bounded by $N$. Since the
number of $k$-summation terms in the state-sum \eqref{eq.fnk} is
bounded by a polynomial function of $N$, this completes the proof of Theorem
\ref{thm.expbound}.
\qed

\begin{remark}
\lbl{rem.expbound2}
The above proof is similar in spirit with the proof of \cite[Thm.15]{GL1}.
\end{remark}

\subsection{Proof of Theorem \ref{thm.Gfun}}
\lbl{sub.thm1}

Fix a quantum spin network $(\Ga,\ga)$. It follows from the shadow state sum 
in Section \ref{sub.recoupling} and Lemma \ref{lem.trivalent} 
that the standard evaluation $\la \Ga,n\ga\ra(q)$ is a sum of a 
balanced $q$-hypergeometric term. Hence, by the fundamental theorem
of WZ theory (see \cite{WZ1}), it follows that $\la \Ga,n\ga\ra(q)$
is $q$-holonomic. Since $\lla \Ga,n\ga\rra(q) = q^{Q(n\ga)/4}\la \Ga,n\ga \ra(q)$
and $Q(n\ga)$ is periodic, $\lla \Ga,n\ga\rra(q)$ is $q$-holonomic as well.

Fix a complex root of unity $\zeta$, and a natural number $r \geq 0$,
and let $a_n=\lla \Ga,n\ga\rra^{(r)}(\zeta)$. We will apply Theorem
\ref{thm.qholo} to check conditions (a)-(d) of Definition \ref{def.Gfunction}.

\begin{itemize}
\item[(a)] 
$a_n \in \BZ[\zeta]$ as follows from Proposition \ref{prop.1}.
\item[(b)]
(b) Theorem \ref{thm.expbound} shows that the coefficients $\lla \Ga,n\ga\rra^(r)(\zeta)$ of the generating
function are exponentially bounded. The same goes for their Galois conjugates since the action of the Galois group
can only replace the root of unity $\zeta$ by another root of unity, so Theorem \ref{thm.expbound} still applies.
\item[(c)]
Since $a_n \in \BZ[\zeta]$ are algebraic integers, 
their common denominator is $1$.
\item[(d)] 
By theorems \ref{thm.der} and \ref{thm.qholo} the coefficients of the 
power series
are holonomic, hence the series itself is holonomic as well.
\end{itemize}
This concludes the proof of Theorem \ref{thm.Gfun}.
\qed

\section{Examples}
\lbl{sec.examples}

In this section we consider various simple examples to illustrate the theorems. For the sake
of simplicity we surpress the function $A^{-Q(\ga)}$ so we choose to work with $\la \Ga, \ga \ra$, 
rather than the more correct $\lla \Ga, \ga \rra$.

\subsection{Evaluations of the q-multinomial coefficients at a root
of unity}

\begin{lemma}
\lbl{lem.qmultinomial}
%Consider the symmetric quantum multinomial. 
Let $q = e^{\frac{2\pi i}{N}}$ and set $a_i = A_iN+\alpha_i$. 
We have 
\begin{equation}
\qbinom{a_1+\hdots+a_n}{a_1,a_2,\hdots,a_n}(q) 
=  q^{\frac{1}{2}\sum_{j<k}\alpha_j\alpha_k-a_ja_k}\binom{A_1+\hdots+A_n}{A_1,A_2,\hdots,A_n}
\qbinom{\alpha_1+\hdots+\alpha_n}{\alpha_1,\alpha_2,\hdots,\alpha_n}
(q)
\end{equation}
\end{lemma}

This was proven for Gaussian (asymmetric) binomial coefficients in 
\cite{De}. The extension to the multinomial
case is straightforward and the power of $q$ appears when one converts to 
symmetric multinomial coefficients. Note that
the power of $q$ is actually a sign since $\alpha_j\alpha_k-a_ja_k = 
-N(A_jA_kN+\alpha_jA_k+A_j\alpha_k)$.

Now consider the growth rates of the theta evaluations. Fix and admissible 
triple $a,b,c\in\mathbb{N}$, and set $A = \frac{-a+b+c}{2},\ B 
= \frac{a-b+c}{2},\ C = \frac{a+b-c}{2}$ then we have 
 \[\la \Theta,an,bn,cn\ra = 
\la \psdraw{1j}{0.15in},n(A+B+C)\ra\qbinom{n(A+B+C)}{nA,nB,nC}\]
At $q =1$ we find using {\em Stirling's formula} (see \cite{O}) 
$$
\Lambda_1 = 
\left\{(-1)^{A+B+C}\frac{(A+B+C)^{A+B+C}}{A^AB^BC^C}\right\}
$$
At the $N$--th root of unity we find 
$$
\la \Theta,an,bn,cn\ra(e^{\frac{2\pi i}{N}}) 
= P(n)\binom{\lfloor\frac{ nA}{N}\rfloor
+\lfloor\frac{ nB}{N}\rfloor
+\lfloor\frac{ nC}{N}\rfloor}{\lfloor\frac{ nA}{N}\rfloor,
\lfloor\frac{ nB}{N}\rfloor,\lfloor\frac{ nC}{N}\rfloor}
$$
where $P(n)$ is a periodic function in $n$ of period $2N$. Another 
application of Stirling
shows that $\Lambda_N \subset \Omega_{2N}|\Lambda_1|^{\frac{1}{N}}$,
where $\Omega_N$ denotes the set of all $N$-th roots of unity.

\subsection{The $q$-difference equation of the regular 6j-symbol}
\lbl{sub.recq6j}

We compute a second order recursion relation 
for the regular quantum $6j$-symbol 
$$
a_n(q)= \la  \psdraw{tetra}{0.18in},n \, \ga \ra(q), \qquad \ga=(2,2,2,2,2,2).
$$
using the $q$-WZ method, implemented in {\em Mathematica} by \cite{PaRi}
and used as in \cite{GV}. 
%%% see Mathematica file: QSixJRecursion.nb
%%% It was typed as follows: select text in Mathemarica->copy as->latex
%%% then paste to kwrite file output.txt and then import into the latex file
The recursion has the form
\begin{equation}
\lbl{eq.reg6jrec}
\sum_{k=0}^2 c_k(q,q^n) a_{n-k}(q)=0
\end{equation}
where
{\small  %tiny
\begin{eqnarray*}
c_0(q,q^n)&=&
q^{2+12 n} \left(q-q^n\right)^3 \left(-1+q^n\right)^5 \left(q+q^n\right) 
\left(4 q^{10}+4 q^{10 n}+6 q^{9+n}+4 q^{8+2 n}-q^{6+3 n}+5 q^{7+3 n}-q^{8+3 n}
\right. \\ & & \left.
-6 q^{5+4 n}-2 q^{6+4 n}-6 q^{7+4 n}-4 q^{4+5 n}+2 q^{5+5 n}-4 q^{6+5 n}-6 q^{3+6 n}
-2 q^{4+6 n}-6 q^{5+6 n}-q^{2+7 n}
\right. \\ & & \left.
+5 q^{3+7 n}-q^{4+7 n}+4 q^{2+8 n}+6 q^{1+9 n}\right)
\end{eqnarray*}
}
{\small  %tiny
\begin{eqnarray*}
c_1(q,q^n)&=&
-q^{-1+6 n} \left(q-q^n\right)^3 \left(q-q^{2 n}\right) \left(-4 q^{14}
-4 q^{15}+q^{18 n}+7 q^{19 n}+10 q^{20 n}+10 q^{21 n}+7 q^{22 n}
\right. \\ & & \left.
+q^{23 n}-10 q^{13+n}-8 q^{14+n}-10 q^{15+n}-10 q^{12+2 n}+6 q^{13+2 n}+6 q^{14+2 n}
-10 q^{15+2 n}+q^{10+3 n}
\right. \\ & & \left.
-9 q^{11+3 n}+q^{12+3 n}+28 q^{13+3 n}+q^{14+3 n}-9 q^{15+3 n}+q^{16+3 n}+7 q^{9+4 n}
+4 q^{10+4 n}+14 q^{11+4 n}
\right. \\ & & \left.
+13 q^{12+4 n}+13 q^{13+4 n}+14 q^{14+4 n}+4 q^{15+4 n}+7 q^{16+4 n}+10 q^{8+5 n}
+2 q^{9+5 n}+22 q^{10+5 n}
\right. \\ & & \left.
+36 q^{11+5 n}+36 q^{13+5 n}+22 q^{14+5 n}+2 q^{15+5 n}+10 q^{16+5 n}+10 q^{7+6 n}
-3 q^{8+6 n}-16 q^{9+6 n}+4 q^{10+6 n}
\right. \\ & & \left.
-17 q^{11+6 n}-17 q^{12+6 n}+4 q^{13+6 n}-16 q^{14+6 n}-3 q^{15+6 n}+10 q^{16+6 n}
+7 q^{6+7 n}-3 q^{7+7 n}+q^{8+7 n}
\right. \\ & & \left.
+6 q^{9+7 n}-3 q^{10+7 n}+q^{11+7 n}-3 q^{12+7 n}+6 q^{13+7 n}+q^{14+7 n}-3 q^{15+7 n}
+7 q^{16+7 n}+q^{5+8 n}
\right. \\ & & \left.
-18 q^{6+8 n}-31 q^{7+8 n}-7 q^{8+8 n}-47 q^{9+8 n}-37 q^{10+8 n}-37 q^{11+8 n}
-47 q^{12+8 n}-7 q^{13+8 n}-31 q^{14+8 n}
\right. \\ & & \left.
-18 q^{15+8 n}+q^{16+8 n}-8 q^{5+9 n}-18 q^{6+9 n}+4 q^{7+9 n}-18 q^{8+9 n}-12 q^{9+9 n}
-8 q^{10+9 n}-12 q^{11+9 n}
\right. \\ & & \left.
-18 q^{12+9 n}+4 q^{13+9 n}-18 q^{14+9 n}-8 q^{15+9 n}-9 q^{4+10 n}-8 q^{5+10 n}+9 q^{6+10 n}
-7 q^{7+10 n}+13 q^{8+10 n}
\right. \\ & & \left.
+13 q^{9+10 n}+13 q^{10+10 n}+13 q^{11+10 n}-7 q^{12+10 n}+9 q^{13+10 n}-8 q^{14+10 n}
-9 q^{15+10 n}-4 q^{3+11 n}
\right. \\ & & \left.
-2 q^{4+11 n}-q^{5+11 n}-13 q^{6+11 n}-24 q^{7+11 n}+2 q^{8+11 n}-43 q^{9+11 n}+2 q^{10+11 n}
-24 q^{11+11 n}-13 q^{12+11 n}
\right. \\ & & \left.
-q^{13+11 n}-2 q^{14+11 n}-4 q^{15+11 n}+11 q^{3+12 n}+31 q^{4+12 n}+21 q^{5+12 n}+44 q^{6+12 n}
+53 q^{7+12 n}
\right. \\ & & \left.
+56 q^{8+12 n}+56 q^{9+12 n}+53 q^{10+12 n}+44 q^{11+12 n}+21 q^{12+12 n}+31 q^{13+12 n}
+11 q^{14+12 n}-4 q^{2+13 n}
\right. \\ & & \left.
+18 q^{3+13 n}-12 q^{4+13 n}-14 q^{5+13 n}-26 q^{6+13 n}-40 q^{7+13 n}-36 q^{8+13 n}
-40 q^{9+13 n}-26 q^{10+13 n}
\right. \\ & & \left.
-14 q^{11+13 n}-12 q^{12+13 n}+18 q^{13+13 n}-4 q^{14+13 n}+11 q^{2+14 n}+31 q^{3+14 n}
+21 q^{4+14 n}+44 q^{5+14 n}
\right. \\ & & \left.
+53 q^{6+14 n}+56 q^{7+14 n}+56 q^{8+14 n}+53 q^{9+14 n}+44 q^{10+14 n}+21 q^{11+14 n}
+31 q^{12+14 n}+11 q^{13+14 n}
\right. \\ & & \left.
-4 q^{1+15 n}-2 q^{2+15 n}-q^{3+15 n}-13 q^{4+15 n}-24 q^{5+15 n}+2 q^{6+15 n}-43 q^{7+15 n}
+2 q^{8+15 n}-24 q^{9+15 n}
\right. \\ & & \left.
-13 q^{10+15 n}-q^{11+15 n}-2 q^{12+15 n}-4 q^{13+15 n}-9 q^{1+16 n}-8 q^{2+16 n}+9 q^{3+16 n}
-7 q^{4+16 n}+13 q^{5+16 n}
\right. \\ & & \left.
+13 q^{6+16 n}+13 q^{7+16 n}+13 q^{8+16 n}-7 q^{9+16 n}+9 q^{10+16 n}-8 q^{11+16 n}-9 q^{12+16 n}
-8 q^{1+17 n}-18 q^{2+17 n}
\right. \\ & & \left.
+4 q^{3+17 n}-18 q^{4+17 n}-12 q^{5+17 n}-8 q^{6+17 n}-12 q^{7+17 n}-18 q^{8+17 n}+4 q^{9+17 n}
-18 q^{10+17 n}-8 q^{11+17 n}
\right. \\ & & \left.
-18 q^{1+18 n}-31 q^{2+18 n}-7 q^{3+18 n}-47 q^{4+18 n}-37 q^{5+18 n}-37 q^{6+18 n}-47 q^{7+18 n}
-7 q^{8+18 n}-31 q^{9+18 n}
\right. \\ & & \left.
-18 q^{10+18 n}+q^{11+18 n}-3 q^{1+19 n}+q^{2+19 n}+6 q^{3+19 n}-3 q^{4+19 n}+q^{5+19 n}
-3 q^{6+19 n}+6 q^{7+19 n}+q^{8+19 n}
\right. \\ & & \left.
-3 q^{9+19 n}+7 q^{10+19 n}-3 q^{1+20 n}-16 q^{2+20 n}+4 q^{3+20 n}-17 q^{4+20 n}-17 q^{5+20 n}
+4 q^{6+20 n}-16 q^{7+20 n}
\right. \\ & & \left.
-3 q^{8+20 n}+10 q^{9+20 n}+2 q^{1+21 n}+22 q^{2+21 n}+36 q^{3+21 n}+36 q^{5+21 n}+22 q^{6+21 n}
+2 q^{7+21 n}+10 q^{8+21 n}
\right. \\ & & \left.
+4 q^{1+22 n}+14 q^{2+22 n}+13 q^{3+22 n}+13 q^{4+22 n}+14 q^{5+22 n}+4 q^{6+22 n}+7 q^{7+22 n}
-9 q^{1+23 n}+q^{2+23 n}
\right. \\ & & \left.
+28 q^{3+23 n}+q^{4+23 n}-9 q^{5+23 n}+q^{6+23 n}-10 q^{1+24 n}+6 q^{2+24 n}+6 q^{3+24 n}
-10 q^{4+24 n}-10 q^{1+25 n}
\right. \\ & & \left.
-8 q^{2+25 n}-10 q^{3+25 n}-4 q^{1+26 n}-4 q^{2+26 n}\right)
\end{eqnarray*}
}
{\small  %tiny
\begin{eqnarray*}
c_2(q,q^n)&=&
\frac{1}{q^8}\left(1+q^n\right) \left(q^2-q^{3 n}\right)^4 \left(q^4-q^{3 n}
\right)^4 \left(q^2+q^{2 n}+q^{1+n}\right)^4 \left(-4 q+q^{3 n}+6 q^{4 n}+4 q^{5 n}
+6 q^{6 n}
+q^{7 n}
\right. \\ & & \left.
-6 q^{1+n}-4 q^{1+2 n}-5 q^{1+3 n}+q^{2+3 n}+2 q^{1+4 n}+6 q^{2+4 n}-2 q^{1+5 n}+4 q^{2+5 n}
+2 q^{1+6 n}+6 q^{2+6 n}
\right. \\ & & \left.
-5 q^{1+7 n}+q^{2+7 n}-4 q^{1+8 n}-6 q^{1+9 n}-4 q^{1+10 n}\right)
\end{eqnarray*}
}
The above recursion can be used to deduce a recursion relation
for the sequence $a_n(\zeta)$ at every fixed root of unity $\zeta$, 
following and illustrating 
the proof of Theorem \ref{thm.qholo}. For example, to obtain the
recursion of the sequence $a_n(1)$, use:
\begin{eqnarray*}
c_0(q)&=& -4 ((-1 + n)^3 n^5 (106 - 230 n + 115 n^2)) (q - 1)^{10} + O((q-1)^{11})
\\
c_1(q)&=& 4 (-1 + n)^3 (-1 + 2 n) (-864 + 4470 n + 2114 n^2 - 51003 n^3 + 
   120089 n^4 - 113505 n^5 
\\ & & + 37835 n^6) (q - 1)^{10} + O((q-1)^{11})
\\
c_2(q)&=& -324 ((8 - 18 n + 9 n^2)^4 (-9 + 115 n^2)) (q - 1)^{10}
  + O((q-1)^{11})
\end{eqnarray*}
An alternative method to obtain the recursion relation of the sequence
$a_n(\zeta)$, using the WZ method, is described in the next section.

\subsection{The regular 6j-symbol at roots of unity}
\lbl{sub.qmultinomial}

The regular $6j$-symbol corresponds to the labeling $\ga=(2,2,2,2,2,2)$ and set
$$
a_n=\la  \psdraw{tetra}{0.18in},n \, \ga \ra(1), \qquad\qquad
b_n=\la  \psdraw{tetra}{0.18in},n \, \ga \ra(-1).
$$
Equation \eqref{eq.6jsum} gives 
$$
a_n
=
\sum_{k=3n}^{4n} (-1)^k\frac{(k+1)!}{(k-3n)!^4(4n-k)!^3}.
$$
Using the WZ method (implemented in {\tt Mathematica} by \cite{PaRi} 
and using it as in \cite{GV}) it follows 
that $(a_n)$ satisfies the recursion relation
$$
\psdraw{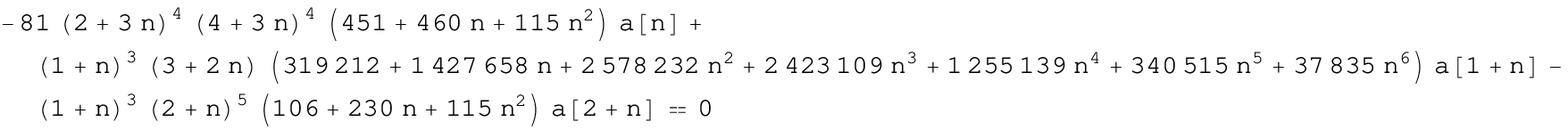}{6in}
$$
%%%%% Use Mathematics file Examples6j.nb and select an output and 
%%%%% Edit->Save Selection As-> select .eps 
This linear recursion has two complex conjugate 
formal power series solutions $a_{\pm,n}$ where
\begin{eqnarray*}
a_{+, n} &=& \frac{1}{n^{3/2}} (329 + 460 i \sqrt{2})^n \left(
1 
+ \frac{-304 - 31 i \sqrt{2}}{ 576 n} 
+ \frac{25879 + 18352 i \sqrt{2}}{331776 n^2}
+ \frac{71176912 + 6323071 i \sqrt{2}}{573308928 n^3}
\right. \\ & & \left.
+ \frac{5 (-2742864803 - 10265264480 i \sqrt{2})}{660451885056 n^4} 
+ \frac{-82823449457840 - 12750219708659 i \sqrt{2}}{380420285792256 n^5} 
\right. \\ & & \left.
+ \frac{-61273664686901989 + 213495822835779152 i \sqrt{2}}{
   657366253849018368 n^6}+ O\left(\frac{1}{n^7}\right) \right)
\end{eqnarray*}
%$$
%\psdraw{math2}{6in}
%$$
When $q=-1$, Equation \eqref{eq.6jsum} and Lemma \ref{lem.qmultinomial} 
imply that $b_n=0$ for odd $n$, and 
$$
b_{2n}=
\sum_{k=3n}^{4n} (-1)^k\frac{k!}{(k-3n)!^4(4n-k)!^3}.
$$
It follows that the sequence $(b_{2n})$ satisfies the following
recursion relation
$$
\psdraw{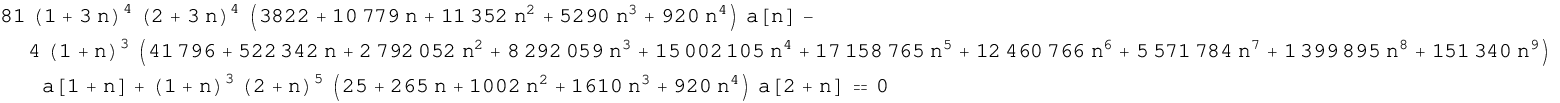}{6in}
$$
This linear recursion has two complex conjugate 
formal power series solutions $b_{\pm,2n}$ where
\begin{eqnarray*}
b_{+,2n} &=& \frac{1}{n^{5/2}} (329 + 460 i \sqrt{2})^n \left(
1
+ \frac{-472 - 19 i \sqrt{2}}{576 n}
+ \frac{105199 + 20632 i \sqrt{2}}{331776 n^2}
+ \frac{22386472 - 9304565 i \sqrt{2}}{573308928 n^3} 
\right. \\ & & \left.
+ \frac{-30711007135 - 48640734448 i \sqrt{2}}{660451885056 n^4}
+ \frac{-80744339543960 + 2822061829369 i \sqrt{2}}{380420285792256 n^5}
\right. \\ & & \left.
+ \frac{17678244315725891 + 219394408835134568 i 
\sqrt{2}}{657366253849018368 n^6}
+  O\left(\frac{1}{n^7}\right) \right)
\end{eqnarray*}

%Note that the observed growth rate is again in agreement with 
%Conjecture \ref{conj.1}.

\section{Open problems}
\lbl{sec.open}

Fix a cubic ribbon graph $\Ga$ with edge set $E(\Ga)$, a complex root of 
unity $\zeta$, and consider the {\em generating series}
$$
S_{\Ga,\zeta}=
 \sum_{\ga} \lla\Ga, \ga \rra(\zeta) \prod_{e \in E(\Ga)} 
z_e^{\ga_e} \in \BQ(\zeta)[[z_{e}, e \in E(\Ga)]]
$$ 
In \cite[Thm.5]{GV}, we proved that when $\zeta=1$, 
$S_{\Ga,1}$ is a rational function, i.e., belongs to the field 
$\BQ(z_{e}, e \in E(\Ga))$. It was explained in \cite[Thm.6]{GV}
that the generating series $F_{\Ga,\ga,\zeta}$ \eqref{eq.FGa} is a diagonal of 
$S_{\Ga,\zeta}$, and consequently the rationality of $S_{\Ga,\zeta}$ 
implies that 
$F_{\Ga,\ga,\zeta}$ is a $G$-function. The rationality of $S_{\Ga,\zeta}$ 
follows
from the so-called {\em chromatic evaluation} of classical spin networks.
For details, see \cite{GV} and also \cite{We}. The chromatic evaluation
seems to break down in the case of complex roots of unity.

\begin{question}
\lbl{que.1}
Is it true that for all cubic ribbon graphs $\Ga$ and all complex
roots of unity, $S_{\Ga,\zeta}$ is a rational function, i.e., belongs to
the field $\BQ(\zeta)(z_{e}, e \in E(\Ga))$?
\end{question}

\bibliographystyle{hamsalpha}\bibliography{biblio}

\providecommand{\bysame}{\leavevmode\hbox to3em{\hrulefill}\thinspace}
\providecommand{\href}[2]{#2}
\providecommand{\eprint}{\begingroup \urlstyle{rm}\Url}
\begin{thebibliography}{GvdV09}

\bibitem[AB06]{AB}
S.~A. Abramov and M.~Bronshte{\u\i}n, \emph{Solution of linear differential and
  difference systems with respect to some of the unknowns}, Zh. Vychisl. Mat.
  Mat. Fiz. \textbf{46} (2006), no.~2, 229--241.

\bibitem[And98]{Aw}
George~E. Andrews, \emph{The theory of partitions}, Cambridge Mathematical
  Library, Cambridge University Press, Cambridge, 1998, Reprint of the 1976
  original.

\bibitem[Bar03]{Ba}
John~W. Barrett, \emph{Geometrical measurements in three-dimensional quantum
  gravity}, Proceedings of the {T}enth {O}porto {M}eeting on {G}eometry,
  {T}opology and {P}hysics (2001), vol.~18, 2003, pp.~97--113.

\bibitem[BC98]{BC}
John~W. Barrett and Louis Crane, \emph{Relativistic spin networks and quantum
  gravity}, J. Math. Phys. \textbf{39} (1998), no.~6, 3296--3302.

\bibitem[BL81a]{BL1}
L.~C. Biedenharn and J.~D. Louck, \emph{Angular momentum in quantum physics},
  Encyclopedia of Mathematics and its Applications, vol.~8, Addison-Wesley
  Publishing Co., Reading, Mass., 1981, Theory and application, With a foreword
  by Peter A. Carruthers.

\bibitem[BL81b]{BL2}
Lawrence~C. Biedenharn and James~D. Louck, \emph{The {R}acah-{W}igner algebra
  in quantum theory}, Encyclopedia of Mathematics and its Applications, vol.~9,
  Addison-Wesley Publishing Co., Reading, Mass., 1981, With a foreword by Peter
  A. Carruthers, With an introduction by George W. Mackey.

\bibitem[CFS95]{CFS}
J.~Scott Carter, Daniel~E. Flath, and Masahico Saito, \emph{The classical and
  quantum 6{$j$}-symbols}, Mathematical Notes, vol.~43, Princeton University
  Press, Princeton, NJ, 1995.

\bibitem[Cos09]{Co}
Francesco Costantino, \emph{Integrality of kauffman brackets of trivalent
  graphs}, arXiv/0908.0542, 2009.

\bibitem[D{\'e}s82]{De}
Jacques D{\'e}sarm{\'e}nien, \emph{Un analogue des congruences de {K}ummer pour
  les {$q$}-nombres d'{E}uler}, European J. Combin. \textbf{3} (1982), no.~1,
  19--28.

\bibitem[EPR08]{EPR}
Jonathan Engle, Roberto Pereira, and Carlo Rovelli, \emph{Flipped spinfoam
  vertex and loop gravity}, Nuclear Phys. B \textbf{798} (2008), no.~1-2,
  251--290.

\bibitem[FK97]{FK}
Igor~B. Frenkel and Mikhail~G. Khovanov, \emph{Canonical bases in tensor
  products and graphical calculus for {$U_q(\mathfrak{sl}_2)$}}, Duke Math. J.
  \textbf{87} (1997), no.~3, 409--480.

\bibitem[FS09]{FS}
Philippe Flajolet and Robert Sedgewick, \emph{Analytic combinatorics},
  Cambridge University Press, Cambridge, 2009.

\bibitem[Gar09]{Ga3}
Stavros Garoufalidis, \emph{{$G$}-functions and multisum versus holonomic
  sequences}, Adv. Math. \textbf{220} (2009), no.~6, 1945--1955.

\bibitem[Gar10]{Ga4}
\bysame, \emph{What is a sequence of nilsson type?}, arXiv:1009.0276,
  Contemporary Math. in press, 2010.

\bibitem[GL05]{GL1}
Stavros Garoufalidis and Yueheng Lan, \emph{Experimental evidence for the
  volume conjecture for the simplest hyperbolic non-2-bridge knot}, Algebr.
  Geom. Topol. \textbf{5} (2005), 379--403 (electronic).

\bibitem[GvdV09]{GV}
Stavros Garoufalidis and Roland van~der Veen, \emph{Asymptotics of classical
  spin networks}, arXiv:0902.3113, 2009.

\bibitem[Jon87]{J}
V.~F.~R. Jones, \emph{Hecke algebra representations of braid groups and link
  polynomials}, Ann. of Math. (2) \textbf{126} (1987), no.~2, 335--388.

\bibitem[Kat87]{Kz}
Nicholas~M. Katz, \emph{A simple algorithm for cyclic vectors}, Amer. J. Math.
  \textbf{109} (1987), no.~1, 65--70.

\bibitem[KK99]{KK}
Mikhail Khovanov and Greg Kuperberg, \emph{Web bases for {${\rm sl}(3)$} are
  not dual canonical}, Pacific J. Math. \textbf{188} (1999), no.~1, 129--153.

\bibitem[KL94]{KL}
Louis~H. Kauffman and S{\'o}stenes~L. Lins, \emph{Temperley-{L}ieb recoupling
  theory and invariants of {$3$}-manifolds}, Annals of Mathematics Studies,
  vol. 134, Princeton University Press, Princeton, NJ, 1994.

\bibitem[Kup96]{Ku}
Greg Kuperberg, \emph{Spiders for rank {$2$} {L}ie algebras}, Comm. Math. Phys.
  \textbf{180} (1996), no.~1, 109--151.

\bibitem[MV94]{MV}
G.~Masbaum and P.~Vogel, \emph{{$3$}-valent graphs and the {K}auffman bracket},
  Pacific J. Math. \textbf{164} (1994), no.~2, 361--381.

\bibitem[Olv97]{O}
Frank W.~J. Olver, \emph{Asymptotics and special functions}, AKP Classics, A K
  Peters Ltd., Wellesley, MA, 1997, Reprint of the 1974 original [Academic
  Press, New York; MR0435697 (55 \#8655)].

\bibitem[Pen68]{PR}
Roger Penrose, \emph{Semiclassical limit of racah coefficients}, Spectroscopic
  and group theoretical methods in physics, Research Notes in Mathematics,
  North Holland, 1968, pp.~1--58.

\bibitem[Pen71a]{Pe1}
\bysame, \emph{Angular momentum: An approach to combinatorial space time},
  Quantum Theory and Beyond (Ted Bastin, ed.), Research Notes in Mathematics,
  Cambridge University Press, 1971, pp.~151--180.

\bibitem[Pen71b]{Pe2}
\bysame, \emph{Applications of negative dimensional tensors}, Combinatorial
  {M}athematics and its {A}pplications ({P}roc. {C}onf., {O}xford, 1969),
  Academic Press, London, 1971, pp.~221--244.

\bibitem[Per03]{Pz}
Alejandro Perez, \emph{Spin foam models for quantum gravity}, Classical Quantum
  Gravity \textbf{20} (2003), no.~6, R43--R104.

\bibitem[PR]{PaRi}
Peter Paule and Axel Riese, \emph{{\tt qZeil} {M}athematica software},
  \url{http://www.risc.uni-linz.ac.at}.

\bibitem[PWZ96]{PWZ}
Marko Petkov{\v{s}}ek, Herbert~S. Wilf, and Doron Zeilberger, \emph{{$A=B$}}, A
  K Peters Ltd., Wellesley, MA, 1996, With a foreword by Donald E. Knuth, With
  a separately available computer disk.

\bibitem[Rob02]{Rb}
Justin Roberts, \emph{Asymptotics and {$6j$}-symbols}, Invariants of knots and
  3-manifolds ({K}yoto, 2001), Geom. Topol. Monogr., vol.~4, Geom. Topol.
  Publ., Coventry, 2002, pp.~245--261 (electronic).

\bibitem[Rov04]{Ro}
Carlo Rovelli, \emph{Quantum gravity}, Cambridge Monographs on Mathematical
  Physics, Cambridge University Press, Cambridge, 2004, With a foreword by
  James Bjorken.

\bibitem[Tur94]{Tu}
V.~G. Turaev, \emph{Quantum invariants of knots and 3-manifolds}, de Gruyter
  Studies in Mathematics, vol.~18, Walter de Gruyter \& Co., Berlin, 1994.

\bibitem[TV92]{TV}
V.~G. Turaev and O.~Ya. Viro, \emph{State sum invariants of {$3$}-manifolds and
  quantum {$6j$}-symbols}, Topology \textbf{31} (1992), no.~4, 865--902.

\bibitem[vdPS03]{VPS}
Marius van~der Put and Michael~F. Singer, \emph{Galois theory of linear
  differential equations}, Grundlehren der Mathematischen Wissenschaften
  [Fundamental Principles of Mathematical Sciences], vol. 328, Springer-Verlag,
  Berlin, 2003.

\bibitem[VMK88]{VMK}
D.~A. Varshalovich, A.~N. Moskalev, and V.~K. Khersonski{\u\i}, \emph{Quantum
  theory of angular momentum}, World Scientific Publishing Co. Inc., Teaneck,
  NJ, 1988, Irreducible tensors, spherical harmonics, vector coupling
  coefficients, $3nj$ symbols, Translated from the Russian.

\bibitem[Wes98]{We}
Bruce~W. Westbury, \emph{A generating function for spin network evaluations},
  Knot theory ({W}arsaw, 1995), Banach Center Publ., vol.~42, Polish Acad.
  Sci., Warsaw, 1998, pp.~447--456.

\bibitem[Wig41]{Wi}
Eugene~P. Wigner, \emph{On representations of certain finite groups}, Amer. J.
  Math. \textbf{63} (1941), 57--63.

\bibitem[WZ85]{WZ2}
Jet Wimp and Doron Zeilberger, \emph{Resurrecting the asymptotics of linear
  recurrences}, J. Math. Anal. Appl. \textbf{111} (1985), no.~1, 162--176.

\bibitem[WZ92]{WZ1}
Herbert~S. Wilf and Doron Zeilberger, \emph{An algorithmic proof theory for
  hypergeometric (ordinary and ``{$q$}'') multisum/integral identities},
  Invent. Math. \textbf{108} (1992), no.~3, 575--633.

\bibitem[Yok96]{Yo}
Yoshiyuki Yokota, \emph{Topological invariants of graphs in {$3$}-space},
  Topology \textbf{35} (1996), no.~1, 77--87.

\bibitem[Zei90]{Z}
Doron Zeilberger, \emph{A holonomic systems approach to special functions
  identities}, J. Comput. Appl. Math. \textbf{32} (1990), no.~3, 321--368.

\end{thebibliography}
\end{document}